\documentclass[12pt]{amsart}
\usepackage{amsmath,amssymb,amsthm,amscd,verbatim}
\usepackage[dvips]{graphicx}
\theoremstyle{plain}
\newtheorem{thm}{Theorem}[section]
\newtheorem{lem}[thm]{Lemma}

\theoremstyle{definition}

\theoremstyle{remark}
\newtheorem{rem}[thm]{Remark}

\begin{document}
\baselineskip=19pt
\title{Seifert surfaces in open books, and pass moves on links}
\author{Susumu Hirose}
\address{Department of Mathematics,  
Faculty of Science and Technology, 
Tokyo University of Science, 
Noda, Chiba, 278-8510, Japan} 
\email{hirose\b{ }susumu@ma.noda.tus.ac.jp}
\author{Yuusuke Nakashima}
\address{Sanyou Junior High School, 103-1 Nobusue, Himeji city,
Hyogo 670-0966, Japan} 
\thanks{The first-named author is supported by Grant-in-Aid for 
Scientific Research (C) (No. 24540096), 
Japan Society for the Promotion of Science. }
\begin{abstract} 
The flat plumbing basket presentation of a link is introduced by 
Furihata, Hirasawa and Kobayashi.
In this paper, we show that the pass-equivalence and an equivalence 
introduced by using the flat plumbing basket presentation 
are the same relation. 
Furthermore, 
we obtain an evaluation of the minimal number of bands used for the flat 
pluming basket presentation for a knot coming from the degree of the 
Alexander polynomial and the three genus of the knot. 
\end{abstract}
\maketitle

\section{Introduction}

The {\em trivial open book decomposition\/} $\mathcal{O}$ of the 3-sphere $S^3$ 
is a decomposition of $S^3$ into infinitely many disks (called {\em pages\/}) 
sharing their boundaries. 
In \cite{FHK}, a concept of positions of Seifert surface called 
the flat plumbing basket associated to $\mathcal{O}$ are introduced and investigated. 
A Seifert surface of a link is said to be a {\em flat plumbing basket\/} 
if it consists of a single page of $\mathcal{O}$ and finitely many bands that are 
embedded in distinct pages. 
We say that a link $L$ admits a {\em flat plumbing basket presentation\/} 
if there is a flat plumbing basket $F$ such that $\partial F = L$. 
We obtain a flat plumbing basket by attaching bands $b_1, b_2, \ldots, b_n$ 
from the bottom to the top on the positive side of a disk. 
By projecting the bands on the disk, we obtain a {\em flat basket diagram\/} 
$\mathcal{D}$ which is a pair consisting of a disk $D$ and a union of 
properly embedded arcs $a_1, a_2, \ldots, a_n$, where 
$a_i$ is the core of the image of $b_i$. 
We call the subscript $i$ of $a_i$ the {\em label\/} of the arc. 
The following theorem assure that any link is presented by 
a flat plumbing diagram. 
\begin{thm}\cite{FHK}\label{thm:existence}
%
Any link admits a flat plumbing basket presentation. 
\end{thm}
\begin{figure}[hbtp]
\includegraphics[height=1.5cm]{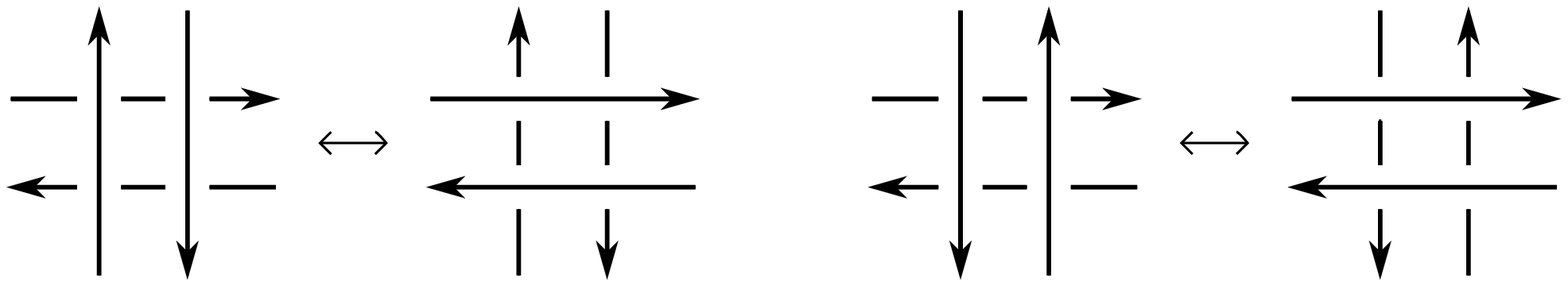}
\caption{
%
}
\label{fig:passmove}
\end{figure}
Let $\underline{\mathcal{D}}$ be a diagram obtained from 
a flat plumbing diagram $\mathcal{D}$ by forgetting the indices. 
We call $\underline{\mathcal{D}}$ a {\em underlying diagram\/} of $\mathcal{D}$. 
Two links $L$ and $L'$ are {\em f.p.b.-equivalent\/} if there is a sequence 
of links $L_0$, $L_1$, $\ldots$, $L_n$ such that 
$L_0 = L$, $L_n=L'$, and there are flat plumbing diagrams $\mathcal{D}_i$ 
and $\mathcal{D}_{i+1}$ of $L_i$ and $L_{i+1}$ respectively such that 
$\underline{\mathcal{D}_i} = \underline{\mathcal{D}_{i+1}}$. 
The local transformation on link diagrams indicated in Figure \ref{fig:passmove} 
is called {\em passmove\/}. 
Two links $L$ and $L'$ are {\em pass-equivalent\/} \cite{Kauffman} \cite{MN} 
if we apply a finite pass moves on $L$ then we obtain $L'$. 
We show: 
\begin{thm}\label{thm:f.p.b.eq}
%
Two links $L$ and $L'$ are f.p.b.-equivalent 
if and only if $L$ and $L'$ are pass-equivalent.   
\end{thm}
We define the {\em flat plumbing basket number\/} of a knot $K$, 
denoted by $fpbk(K)$, to be the minimal number of bands 
to obtain a flat plumbing basket surface of $K$. 
For a flat basket diagram $\mathcal{D}$, by recording the labels of 
the arcs as one travels along $\partial D$ according to counterclockwise 
orientation, one obtain a word $W$ in $\{1,2,\ldots,2n\}$ beginng from 
the letter $1$ such that each letter appears exactly twice. 
We call $W$ a {\em flat basket code\/} for $\mathcal{D}$.  
We introduce a method to obtain the Seifert matrix for the flat plumbing basket, 
and show: 
\begin{thm} \label{thm:lowerbound}
%
Let $K$ be a non-trivial knot, $\Delta_K(t)$ be the Alexander polynomial of $K$, 
$\deg \Delta_K(t) = ($the maximal degree of $\Delta_K(t) )$ $-$ 
$($ the minimal degree of $\Delta_K(t) )$, 
$a$ be the leading coefficient of $\Delta_K(t)$, 
and $g(K)$ be the minimal genus of the Seifert surface (i.e. three genus) 
of $K$. Then $fpbk(K)$ is evaluated as follows: \\
(1) If $a = \pm 1$ then $fpbk(K) \geq 
\max \{ 2g(K)+2, \deg \Delta_K(t) +2 \}$, \\
(2) If $a \not= \pm 1$ then $fpbk(K) \geq 
\max\{ 2g(K)+2, \deg \Delta_K(t) +4 \}$. 
\end{thm}%

\begin{rem}\label{rmk:lowerbound}
(a) 
For a  general knot $K$, $2g(K) \geq \deg \Delta_K(t)$. 
However if  $K$ is alternating \cite{Murasugi} 
or  the minimal crossing number of $K$ 
is at most $10$, then $2 g(K)= \deg \Delta_K(t)$.
Therefore if such a knot has non-monic Alexander polynomial 
as in (2), we have $fpbk(K) \geq \deg \Delta_K(t) + 4 > 2g(K)+2 $.
For example, $g(5_2) = 1$ and $\Delta_{5_2} (t) = -2t^2 + 3 t -2$. 
By the above Theorem 
$fpbk(5_2) \geq \max \{ 4, 6 \}$, namely we found a flat plubming basket for 
$5_2$ whose number of bands is $6$.    \\
(b) The flat plumbing basket number is not additive under connected sum. 
For example, $fpbk(3_1) = fpbk(3_1^*) = 4$, but 
$fpbk(3_1 \# 3_1^*) \leq 6$ as shown on the left of Figure \ref{fig:3-1+3-1}.  
\end{rem}

\newpage

\section{Obtaining a flat plumbing basket presentation \\
from a Seifert surface} 

In \cite{FHK}, Theorem \ref{thm:existence} is proved and  
an algorithm to obtain a flat basket code from a closed braid presentation of a link
is introduced. 
In this section, we obtain a flat plumbing basket presentation 
of a link by using a certain canonical form of the Seifert surface of the link. 

\begin{figure}[hbtp]
\includegraphics[height=2cm]{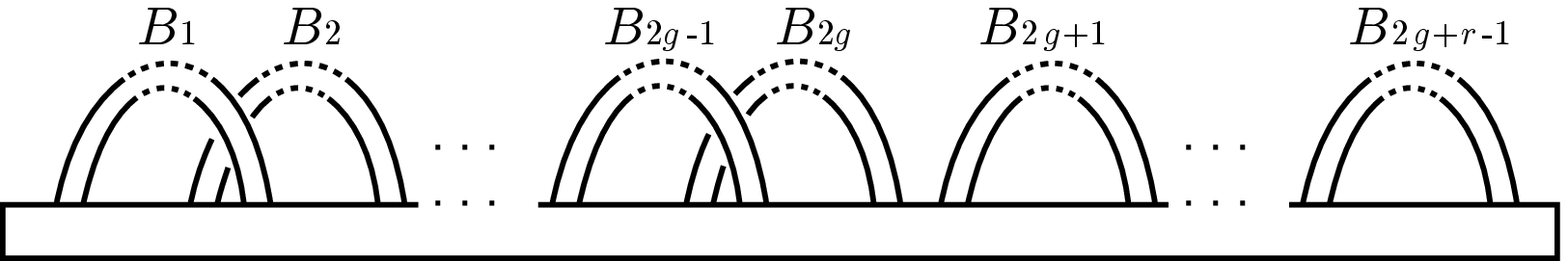}
\caption{
%
}
\label{fig:seifertform}
\end{figure}
\begin{figure}[hbtp]
\includegraphics[height=2cm]{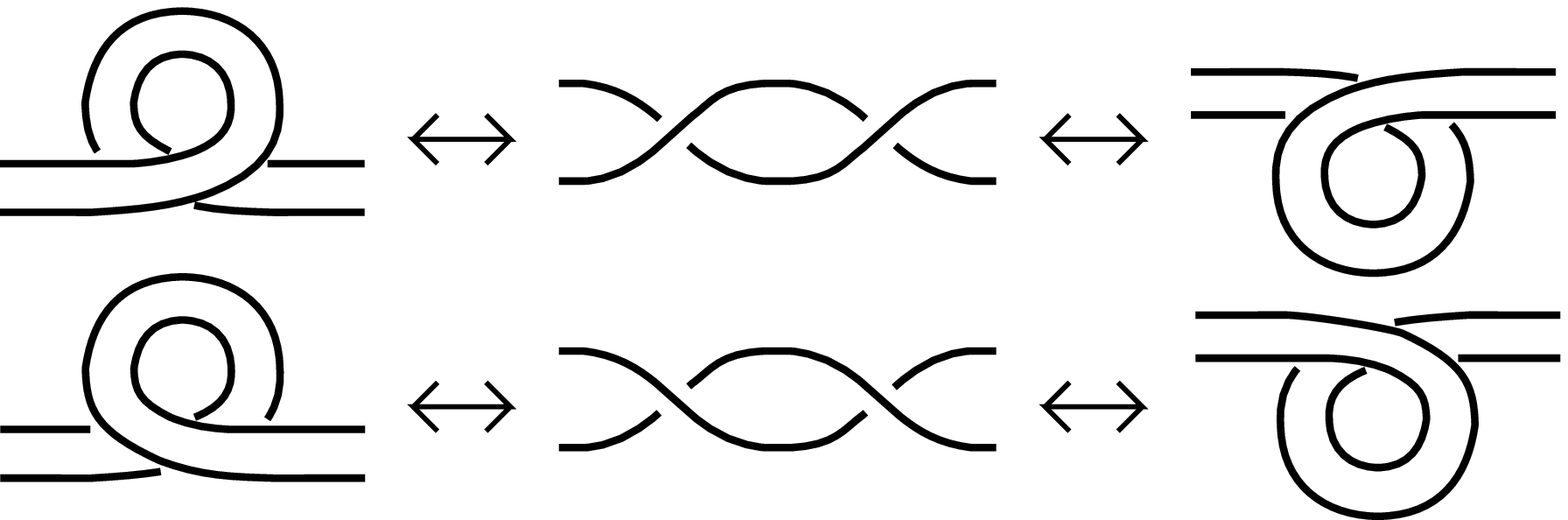}
\caption{
%
}
\label{fig:twistband}
\end{figure}
In general, a Seifert surface $S$ of a link $L$ with $r$ components 
is constructed from a 2-disk $D$ by attaching $2g+r-1$ bands 
$B_1,\ldots, B_{2g+r-1}$ as shown in Figure \ref{fig:seifertform}. 
The bands $B_i$ can be knotted, linked and twisted. 
By deforming the twists of the bands $B_i$ as shown in Figure \ref{fig:twistband}, 
we make the bands $B_i$ untwisted. 
We project this Seifert surface on the $x$-$y$ plane 
such that the image of the 2-disk $D$ is 
$\{ (x,y) \in \mathbb{R}^2 \, | \, 0 \leq x \leq 1, -1 \leq y \leq 0 \}$, 
and, in order to simplify our arguments and figures, 
indicate the bands by the core curves of bands with some normal crossings. 
We call this projection of $S$ the {\em normal form\/} of $S$. 

We isotope $S$ such that the image of its projection onto $x$-$y$ plane satisfy: \\
(1) The image of bands are the union of lines parallel to the $x$-axis 
or $y$-axis. For short, we call the line parallel to the $x$-axis (resp. $y$-axis) 
{\em the $x$-line\/} (resp. {\em the $y$-line\/}). \\
(2) On each crossing of the image of bands, the over-crossing is on 
the $x$-line and the under-crossing is on the $y$-line. \\
(3) The $x$-coordinates of the $y$-lines are distinct, 
and the $y$-coordinates of the $x$-lines are distinct. \\
For (2), we often use the deformation in Figure \ref{fig:cross}. \par
\begin{figure}[hbtp]
\includegraphics[height=3cm]{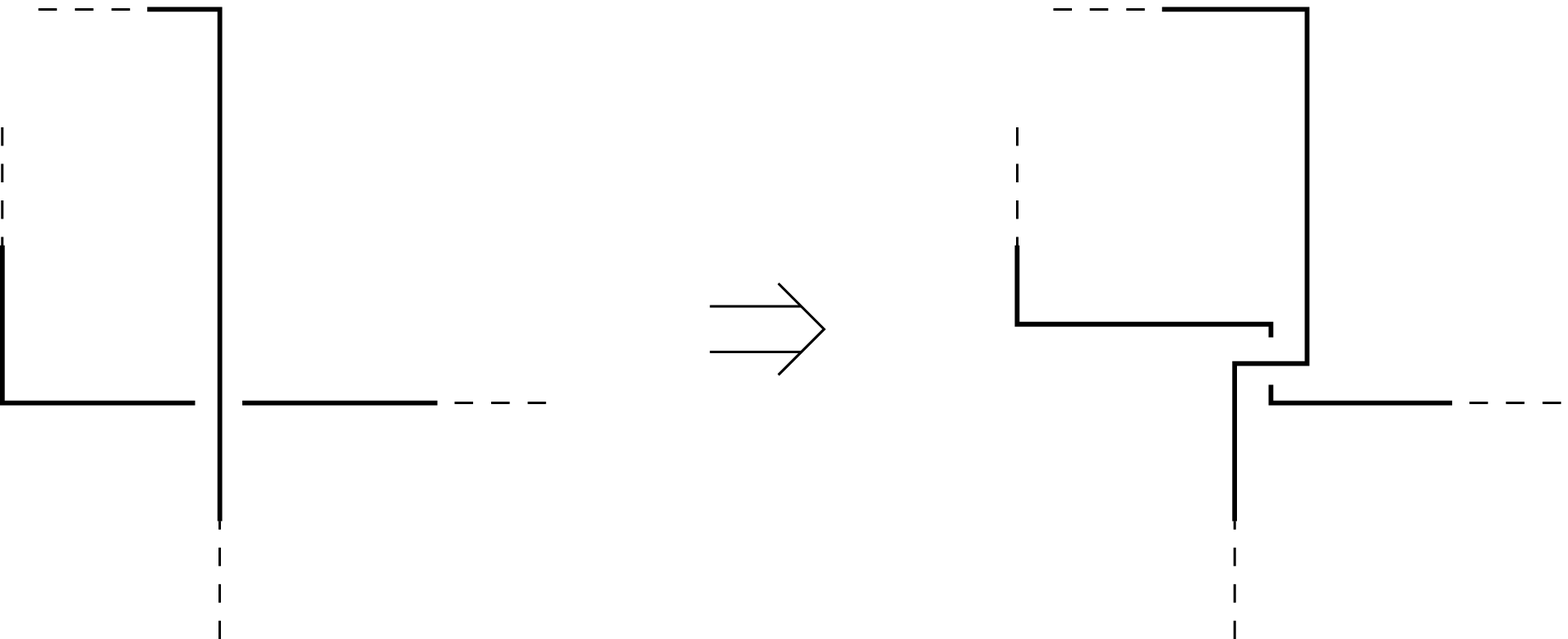}
\caption{
%
}
\label{fig:cross}
\end{figure}
There are four types (a), (b), (c) and (d) of $x$-lines 
according to the positions of adjacent $y$-lines as shown 
in Figure \ref{fig:h-classify}. 
If all $x$-lines are type (a), then this surface $S$ is 
a flat plumbing basket. 
If there are some $x$-lines which are not type (a), 
then we alter $S$ as shown in Figure \ref{fig:pushdown-interval}. 
By this alteration, the genus of $S$ is increased by one, 
but $\partial S = L$ is not changed. 
We call this operation {\em push-down}. 
After we alter all part of $x$-lines where adjacent $y$-lines go above 
by push-downs, 
we get a flat plumbing basket of $L$.  
\begin{figure}[hbtp]
\includegraphics[height=3cm]{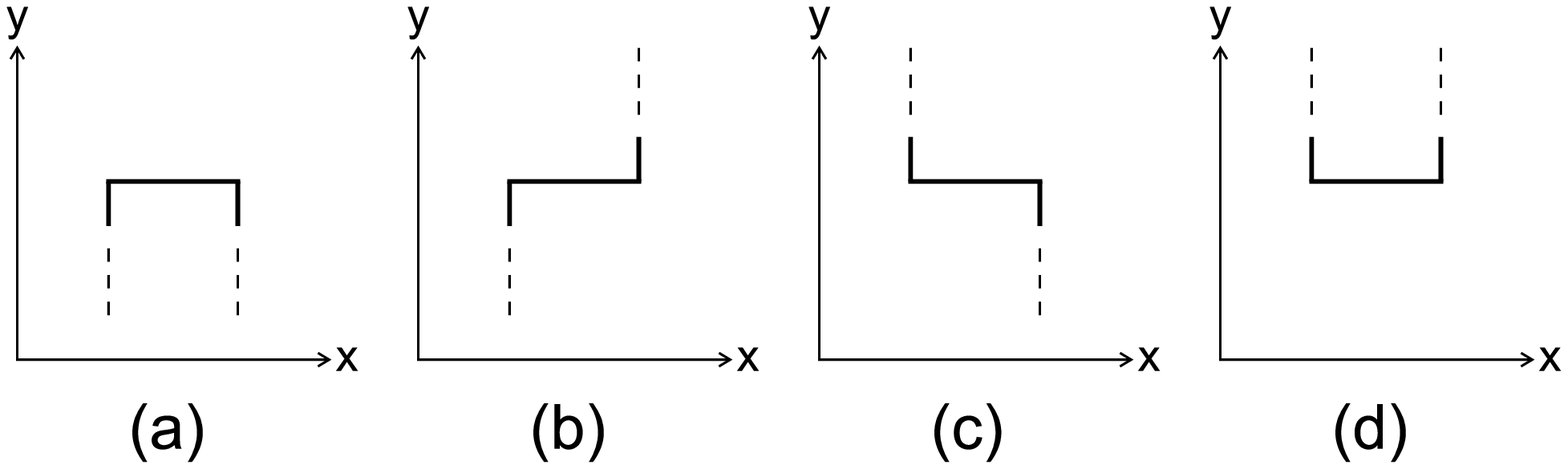}
\caption{
%
}
\label{fig:h-classify}
\end{figure}
\begin{figure}[hbtp]
\includegraphics[height=3cm]{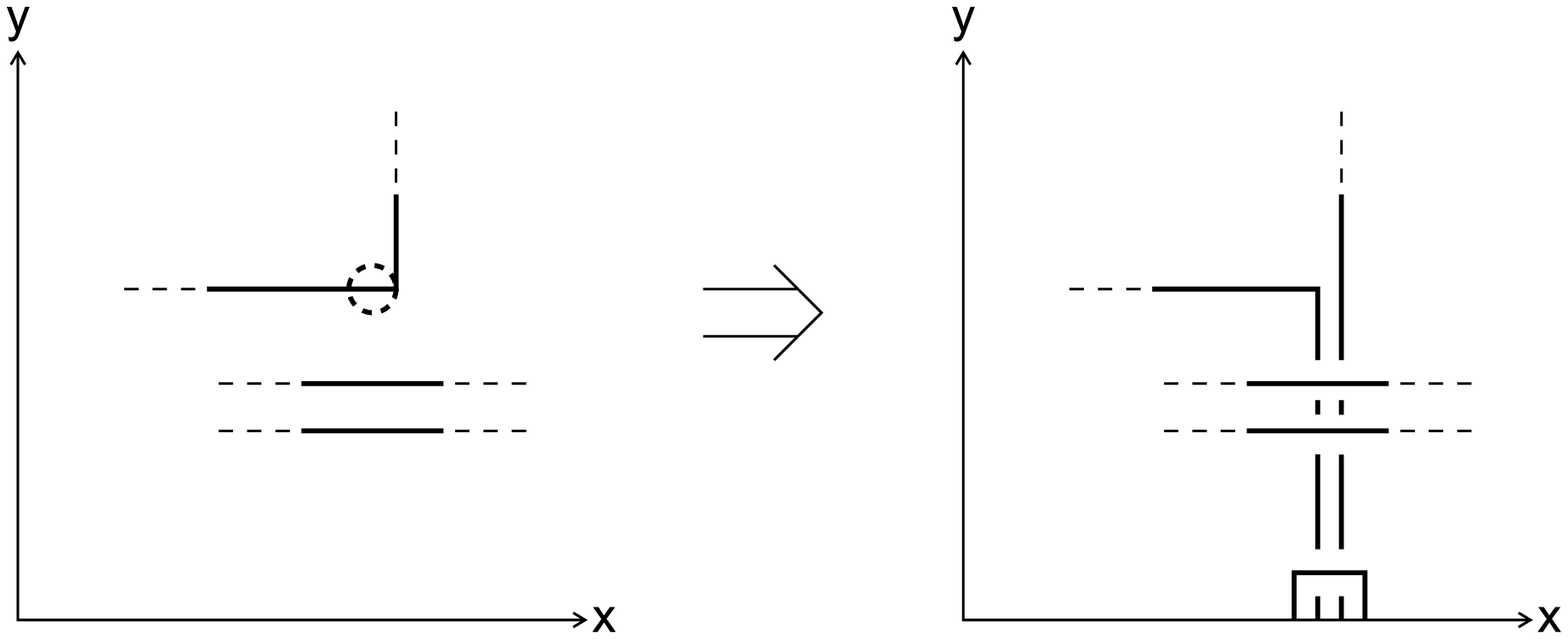}
\caption{
%
}
\label{fig:pushdown-interval}
\end{figure}
\section{Proof of Theorem \ref{thm:f.p.b.eq}} 

For two links $L$ and $L'$, if there are flat basket diagram 
$\mathcal{D}$ and $\mathcal{D}'$ for $L$ and $L'$ respectively 
such that $\underline{\mathcal{D}}=\underline{\mathcal{D}'}$, 
then $L$ and $L'$ are pass-equivalent. 
This means that if 
$L$ and $L'$ are f.p.b.-equivalent then $L$ and $L'$ are pass-equivalent. 

Kauffmann(knot case) \cite{Kauffman}, Murakami and Nakanishi (link case) \cite{MN} 
showed the following theorem: 
\begin{thm}\cite{Kauffman} \cite{MN} \label{thm:KMN} 
%
Any link is pass equivalent to the one of the following links, \\
$I_n$ : a trivial link with $n$ components, \\
$II_n$ : a split sum of a trefoil knot and $I_{n-1}$, \\
$III_{d,n}$ : a link shown in Figure \ref{fig:nonproperlink}. 
\end{thm}
\begin{figure}[hbtp]
\includegraphics[height=2cm]{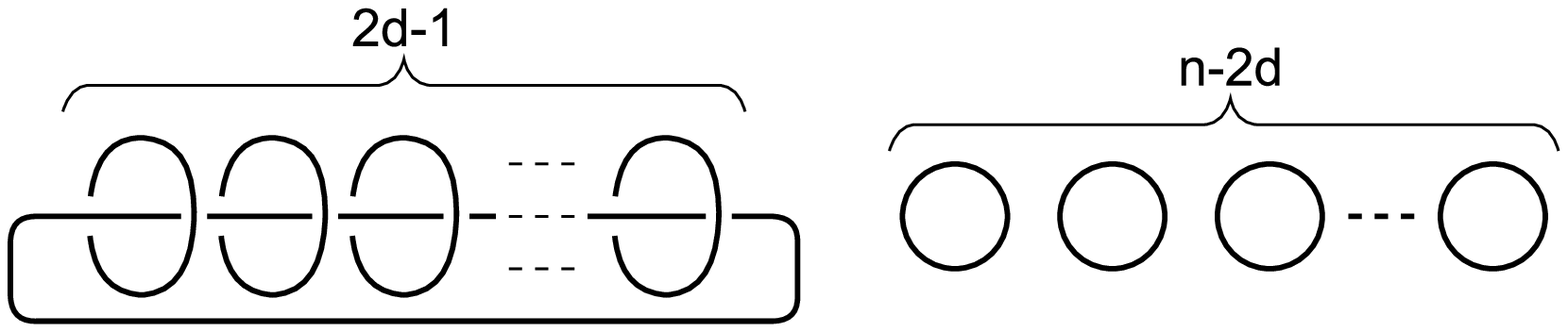}
\caption{
%
}
\label{fig:nonproperlink}
\end{figure}
In order to show that if $L$ and $L'$ are pass-equivalent then 
$L$ and $L'$ are f.p.b.-equivalent, we show: 
\begin{lem} \label{lem:f.p.b.eq-KMN}
%
Any link is f.p.b-equivalent to $I_n$, $II_n$ or $III_{d,n}$. 
\end{lem}

\begin{proof}
\begin{figure}[hbtp]
\includegraphics[height=2cm]{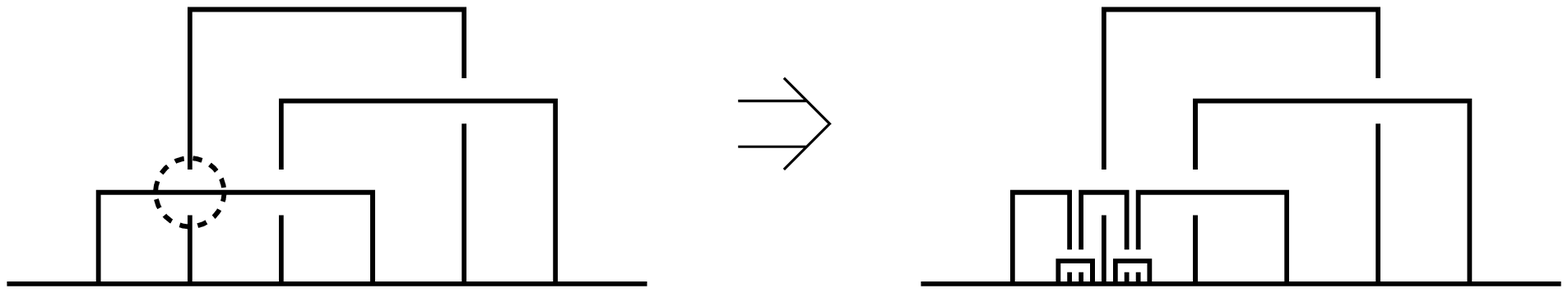}
\caption{
%
}
\label{fig:changeonecross}
\end{figure}
At first, we observe that 
if there are Seifert surfaces $S$ and $S'$ of normal form for $L$ and $L'$ 
respectively such that $S$ is obtained from $S'$ by a crossing change 
at a pair of bands of $S$, then $L$ and $L'$ is f.p.b.-equivalent. 
We deform $S$ such that the projection of $S$ satisfies the conditions 
(1) (2) and (3) in \S 2. 
Let $Q$ be a crossing point where we should change crossing 
in order to obtain $S'$. 
We perform push-downs as explaind in \S 2 and 
obtain a flat plumbing basket $S_1$ for $L$. 
Let $H$ be the $x$-line in the projection image of $S_1$ which contains $Q$. 
We choose points $P_{l_1}, P_{l_2}, P_{r_1}, P_{r_2}$ on $H$ such that 
$P_{l_1}, P_{l_2}, Q, P_{r_1}, P_{r_2}$ are on $H$ in this order from 
the left to the right and there is no crossing between 
$P_{l_1}$ and $P_{r_2}$ except $Q$. 
We push-down the two lines on $H$ between $P_{l_1}, P_{l_2}$ and 
between $P_{r_1}, P_{r_2}$, 
then we get a flat plumbing basket $S_2$ for $L$ 
(see Figure \ref{fig:changeonecross}). 
Furthermore, a surface $S_2'$ obtained by changing the crossing $Q$ is 
a flat plumbing basket for $L'$. 
Therefore, $L$ and $L'$ are f.p.b.-equivalent. 

\begin{figure}[hbtp]
\includegraphics[height=1cm]{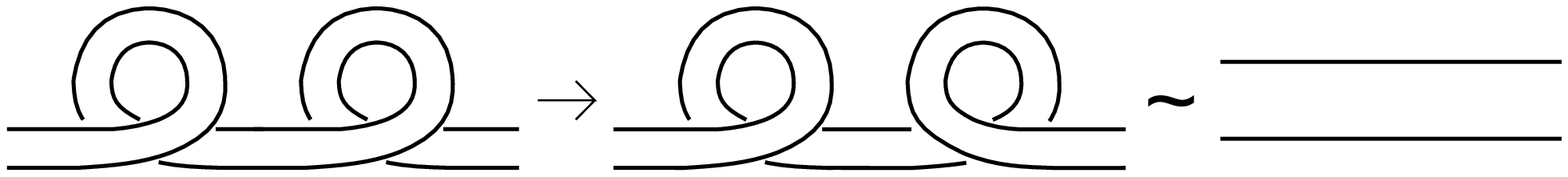}
\caption{
%
}
\label{fig:2-twist}
\end{figure}

Let $L$ be a link and $S$ be a Seifert surface for $L$ of normal form. 
By changing crossings of bands in $S$, bands in $S$ are untangled. 
By changing crossings of bands as is indicated in Figure \ref{fig:2-twist}, 
times of curls of bands in $S$ are changed to $0$ or $1$. 
By the observation in the previous paragraph, we see that 
$L$ is f.p.b.-equivalent to the boundary of the boundary connected sum of 
the sufaces shown in Figure \ref{fig:trivial-or-3-1}. 
\begin{figure}[hbtp]
\includegraphics[height=2cm]{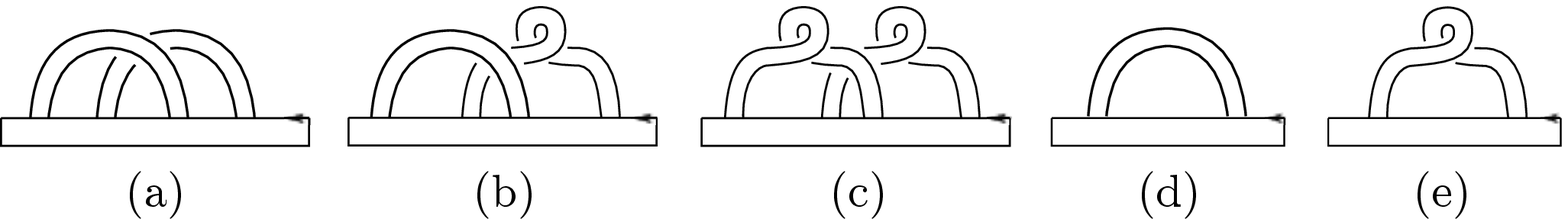}
\caption{
%
}
\label{fig:trivial-or-3-1}
\end{figure}
Since the boundaries of (a) and (b) are trivial knot, 
we remove these part by the isotopy of $L$. 
If there is a pair of (c), then $L$ contains $3_1 \# 3_1$ in its connected 
sum decomposition. 
By changing the crossing of bands in one of the pair of (c), $3_1 \# 3_1$ is altered 
to $3_1 \# 3_1^*$. 
\begin{figure}[hbtp]
\includegraphics[height=2cm]{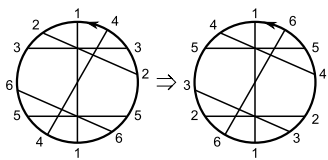}
\caption{
%
$3_1 \# 3_1^*$ is f.p.b.-equivalent to the trivial knot. 
}
\label{fig:3-1+3-1}
\end{figure}
Figure \ref{fig:3-1+3-1} indicate that $3_1 \# 3_1^*$ is f.p.b.-equivalent to 
the trivial knot. Therefore, $L$ is f.p.b.-equivalent to 
one of the links listed in Figure \ref{fig:link-classify}. 
\begin{figure}[hbtp]
\includegraphics[height=7cm]{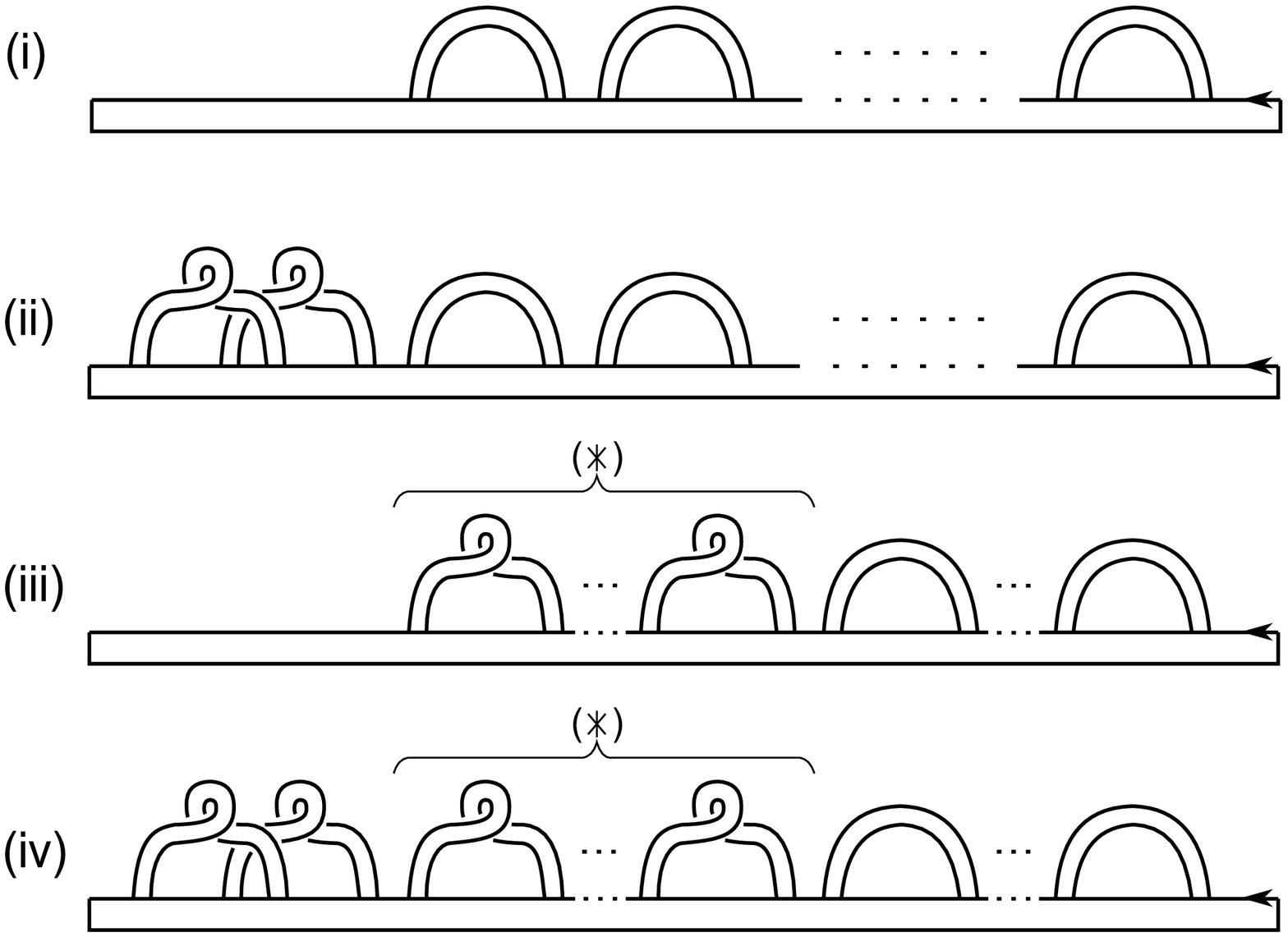}
\caption{
%
}
\label{fig:link-classify}
\end{figure}
Figure \ref{fig:iv-iii} shows that (iv) is f.p.b.-equivalent to (iii). 
Hence, we consider (i) (ii) and (iii). 
\begin{figure}[hbtp]
\includegraphics[height=4cm]{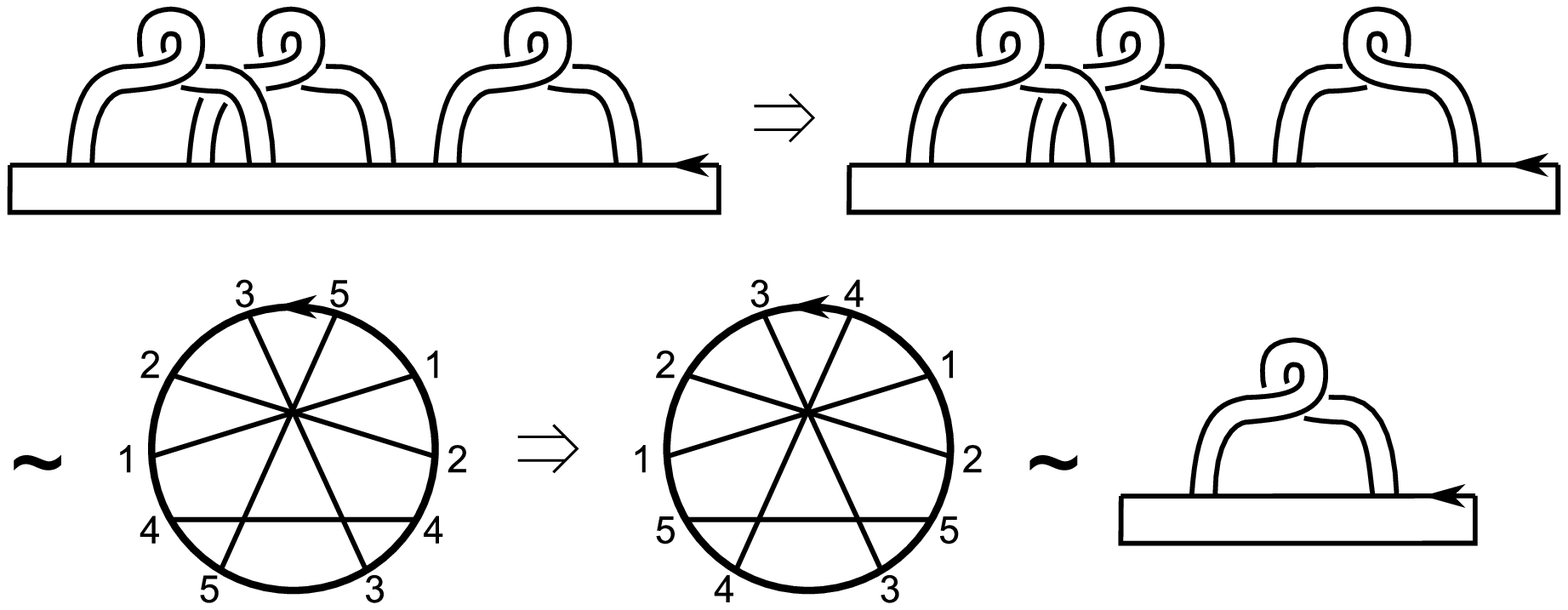}
\caption{
%
}
\label{fig:iv-iii}
\end{figure}
The boundaries of (i) and (ii) are $I_n$ and $II_n$ respectively. 
The boundaries of (iii) is $III_{d,n}$, if (*) is odd. 
If (*) is even, we change crossing of bands as shown in Figure \ref{fig:hopflink} 
and deform as shown in Figure \ref{fig:2d-1}, then (*) is changed to odd. 
\begin{figure}[hbtp]
\includegraphics[height=1.3cm]{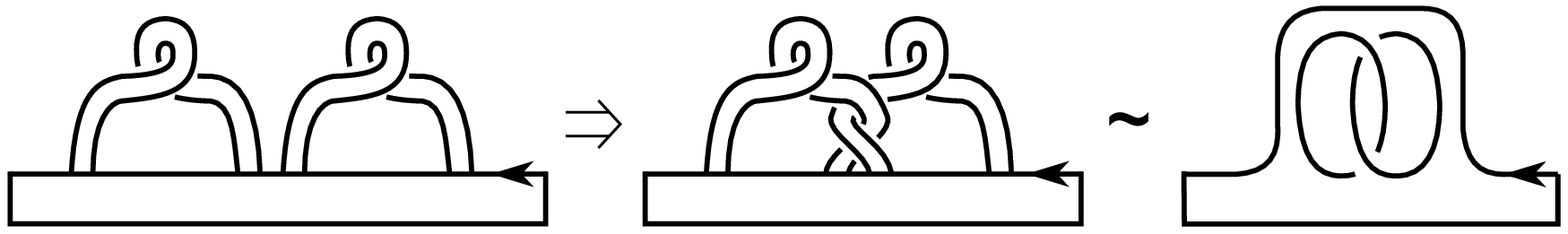}
\caption{
%
}
\label{fig:hopflink}
\end{figure}
\begin{figure}[hbtp]
\includegraphics[height=9cm]{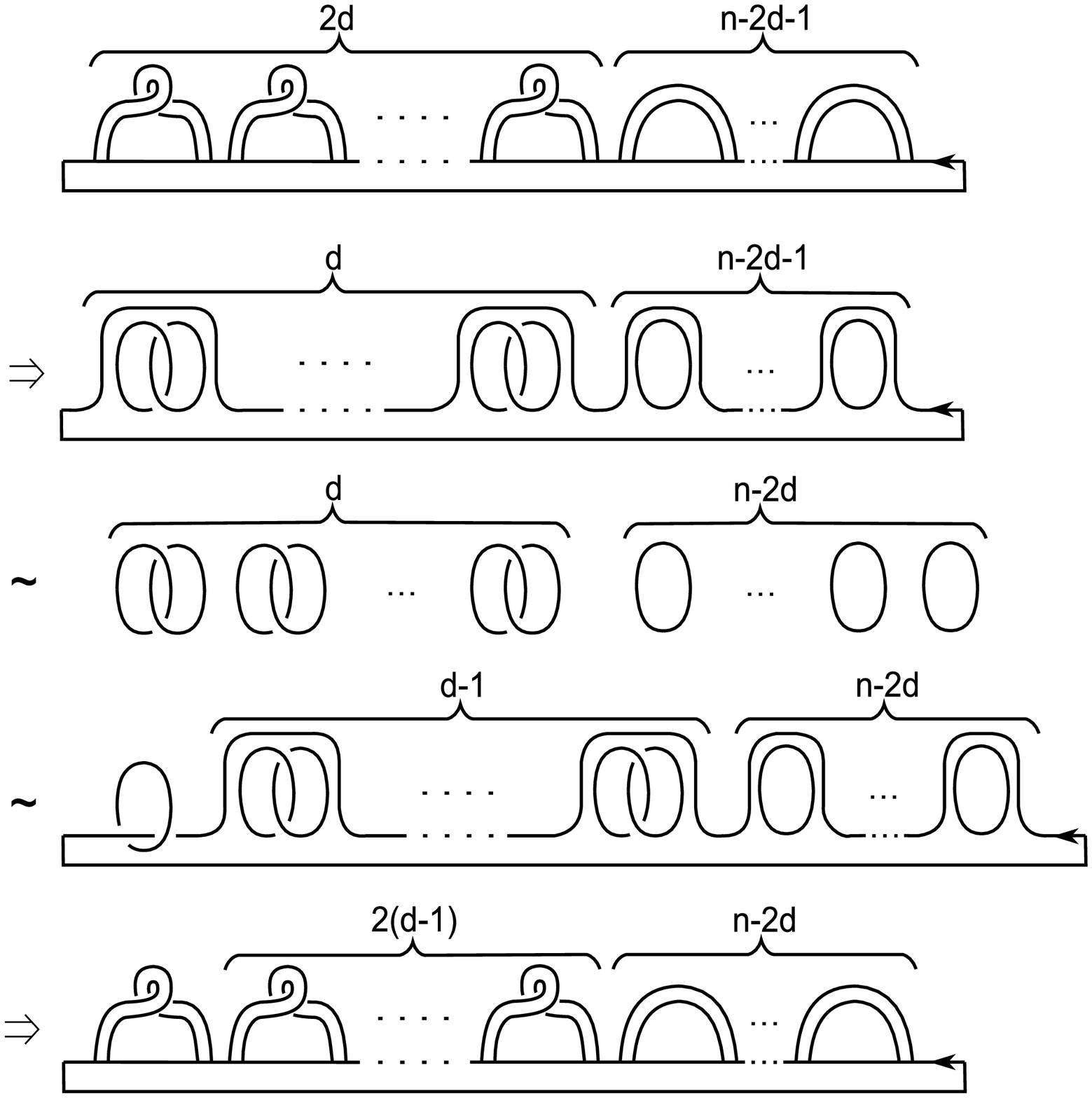}
\caption{
%
}
\label{fig:2d-1}
\end{figure}
\end{proof}

\section{Alexander polynomial 
and flat plumbing basket number} 

In this section, we introduce a method to obtain the Seifert matrix of 
the flat plumbing basket and 
the Alexander polynomial of the link from 
the flat basket code, and prove Theorem \ref{thm:lowerbound} 
on a lower bound for the flat plumbing basket number $fpbk(K)$ of $K$ 
coming from the Alexander polynomial of $K$ and the three genus of $K$. 

We review the definition of the Seifert matrix and the Alexander polynomial. 
Let $S$ be a Seifert surface for a knot $K$, 
and $\{ k_1, \ldots, k_{2n} \}$ be a basis for $H_1(S, \mathbb{Z})$. 
Let $v_{i,j}= Link(k_i^{+},k_j)$, where $k_i^{+}$ is a cycle in $S^3$ obtained by 
pushing $k_i$ into $S^3 \setminus S$ by the positive normal direction of $S$, 
and $Link$ is the linking number. 
The $2n \times 2n$ matrix $V = (v_{i,j})$ is the {\em Seifert matrix} 
of $S$. 
The polynomial $\Delta_K(t) = \det(V - t V^{T})$, 
where $V^{T}$ is the transpose of $V$, 
is the {\em Alexander polynomial} of $K$. 

Let $F$ be the flat plumbing basket whose flat basket code is a word $W$ 
in $\{1,2,\ldots,2n\}$, $p_i$ be a point on $\partial D$ 
corresponding to the former $i$ in $W$ 
and $q_i$ be a point on $\partial D$  corresponding to the latter $i$ in $W$. 
We define a cycle $k_i$ on $F$ to be constructed from the line on $D$ oriented 
from $q_i$ to $p_i$ and a core arc of $b_i$ oriented from $p_i$ to $q_i$. 
Then $\{k_1, \ldots, k_{2n} \}$ is a basis of $H_1(F, \mathbb{Z})$. 
We see: 

\begin{lem} \label{lem:Seifert-from-code}
%
Let $W$ be a flat basket code, $F$ be a flat plumbing basket whose 
flat basket code is $W$, and  $\{k_1, \ldots, k_{2n} \}$ be a basis 
of $H_1(F, \mathbb{Z})$ given above. 
Then the $(i,j)$-entry $v_{i,j}$ of the Seifert matrix $V$ of $F$ is determined 
as follows.  
\\ 
We assume that $i < j$. \\
(1) If $W = ( \cdots i \cdots j \cdots j \cdots i \cdots)$ then 
$v_{i,j}= 0$ and $v_{j,i} = 0$. \\
(2) If $W = ( \cdots i \cdots j \cdots i \cdots j \cdots)$ then 
$v_{i,j}= 0$ and $v_{j,i} = -1$. \\
(3) If $W = ( \cdots i \cdots i \cdots j \cdots j \cdots)$ then 
$v_{i,j}= 0$ and $v_{j,i} = 0$. \\
(4) If $W = ( \cdots j \cdots i \cdots i \cdots j \cdots)$ then 
$v_{i,j}= 0$ and $v_{j,i} = 0$. \\
(5) If $W = ( \cdots j \cdots i \cdots j \cdots i \cdots)$ then 
$v_{i,j}= 0$ and $v_{j,i} = 1$. \\
(6) If $W = ( \cdots j \cdots j \cdots i \cdots i \cdots)$ then 
$v_{i,j}= 0$ and $v_{j,i} = 0$. \qed
\end{lem}

\begin{proof} [Proof of Theorem \ref{thm:lowerbound}. ] 
By the definition of $\Delta_K(t)$, 
$$
\Delta_K(t) =  
\begin{vmatrix}
0 & -v_{21} t & -v_{31} t & \cdots & -v_{2n, 1} t \\
v_{21} & 0 & -v_{32} t & \cdots & -v_{2n,2} t \\
\vdots & \vdots & \vdots & \ddots & \vdots \\
v_{2n, 1} & v_{2n, 2} & v_{2n,3} & \cdots  & 0  
\end{vmatrix}
$$
Let $S_{2n}$ be the symmetric group of degree $2n$, 
$\epsilon(\sigma)$ be the signature of $\sigma \in S_{2n}$, 
and $a_{ij}$ be the $(i,j)$-th entry of $V - t V^T$. 
Then 
$\Delta_K(t) = \Sigma_{\sigma \in S_{2n} }
\epsilon(\sigma) a_{1 \sigma(1)} a_{2 \sigma(2)} \cdots a_{2n \sigma(2n)} $. 
This show that the degree of  $\Delta_K(t)$ is at most $2n-1$ and at least $1$, 
and the term with degree $2n-1$ is 
$- v_{21} v_{32} \cdots v_{2n, 2n-1} v_{2n,1} t^{2n-1}$ 
and the term with degree $1$ is 
$- v_{21} v_{32} \cdots v_{2n, 2n-1} v_{2n,1} t$. 
Since $v_{ij} = 0$ or $\pm 1$, 
the term of $\Delta_K(t)$ with degree $2n-1$ should be $0 \cdot t^{2n-1}$ or 
$\pm 1 \cdot t^{2n-1}$. 
Therefore, if the term of  $\Delta_K(t)$ with maximal degree is not 
$\pm t^k$ then the maximal degree is at most $2n-2$ and 
the minimal degree is at least $2$, hence 
$\deg \Delta_K(t)  \leq 2n-4$. 

On the other hand, the cycle $k_1$ bounds a disk $D_1$ whose interior 
is disjoint from the flat plumbing basket $F$, 
and we can compress $F$ along $D_1$. 
Therefore, we see $2g(K)+2 \leq fpbk(K)$. 
\end{proof}

The following table lists flat basket codes for all the prime knots of 
at most 9 crossings and the evaluation of the flat plumbing basket numbers 
of knots obtained by the flat basket codes and 
Theorem \ref{thm:lowerbound}. 
For reference, the three genera $g(K)$ are on the table (see for example 
a comment on Table 2 in p.254, and F.3 begins from p.270 of \cite{Kawauchi}). 
The numbering of the knots follows that of \cite{Rolfsen}. 
We remark by the bullet $\bullet$ the knot which corresponds to the case (2) 
of Theorem \ref{thm:lowerbound} and 
$\max \{ 2g(K)+2, \deg \Delta_K(t) +4 \}= \deg \Delta_K(t)+4$
(see (c) of Remark \ref{rmk:lowerbound}). 
We remark by the asterisk * that the flat plumbing basket number for $9_{24}$ 
was determined by Mika Aoki, who was an undergraduate student advised by 
Tsuyoshi Kobayashi, and by the double asterisk ** that 
the flat plumbing basket number for $9_{29}$ 
was determined by Mikami Hirasawa. 
\medskip

\begin{center}
\begin{tabular}[h]{cccc}
Knot $K$ & Flat basket code of $K$ & $g(K)$ & $fpbk(K)$ \\
\hline
$3_1$ & (1,2,3,4,1,2,3,4) & 1 & 4 \\
$4_1$ & (1,2,4,3,1,2,4,3) & 1 & 4 \\
$5_1$ & (1,2,3,4,5,6,1,2,3,4,5,6) & 2 & 6 \\
$5_2$ & (1,2,3,5,6,4,5,6,1,2,3,4) & 1 & 6 $\bullet$ \\
$6_1$ & (1,2,3,1,2,4,6,5,3,4,6,5) & 1 & 6 $\bullet$ \\
$6_2$ & (1,2,3,4,6,5,1,2,3,4,6,5) & 2 & 6 \\
$6_3$ & (1,2,3,6,5,4,1,2,3,6,5,4) & 2 & 6 \\
$7_1$ & (1,2,3,4,5,6,7,8,1,2,3,4,5,6,7,8) & 3 & 8 \\
$7_2$ & (1,2,8,1,4,5,6,3,4,5,6,7,2,3,7,8) & 1 & 6 - 8 $\bullet$ \\
$7_3$ & (1,3,2,1,8,7,6,5,4,3,8,7,6,5,4,2) & 2 & 8 $\bullet$ \\
$7_4$ & (1,6,3,1,6,2,7,4,3,2,7,8,5,4,8,5) & 1 & 6 - 8 $\bullet$ \\
$7_5$ & (1,2,3,4,7,8,5,6,7,8,1,2,3,4,5,6) & 2 & 8 $\bullet$ \\
$7_6$ & (1,2,3,5,6,4,1,2,3,5,6,4)& 2 & 6 \\
$7_7$ & (1,2,4,3,6,5,1,2,4,3,6,5) & 2 & 6 \\
$8_1$ & (1,3,8,2,1,5,6,4,5,6,7,3,4,7,8,2) & 1 & 6 - 8 $\bullet$ \\
$8_2$ & (1,2,3,4,5,6,8,7,1,2,3,4,5,6,8,7) & 3 & 8 \\
$8_3$ & (1,2,5,3,1,4,6,2,5,4,6,3) & 1 & 6 $\bullet$ \\
$8_4$ & (1,2,3,4,7,5,1,2,7,8,6,3,4,5,8,6) & 2 & 8 $\bullet$ \\
$8_5$ & (1,2,8,7,6,5,1,4,3,2,8,7,6,5,4,3) & 3 & 8 \\
$8_6$ & (1,2,3,4,5,1,2,6,8,7,3,4,5,6,8,7) & 2 & 8 $\bullet$ \\
$8_7$ & (1,8,7,6,3,4,5,2,1,8,7,6,3,4,5,2) & 3 & 8 \\
$8_8$ & (1,4,3,2,1,8,5,6,7,4,3,8,5,6,7,2) & 2 & 8 $\bullet$ \\
$8_9$ & (1,2,3,4,8,7,6,5,1,2,8,7,6,3,4,5) & 3 & 8 \\
$8_{10}$ & (1,2,3,8,7,6,1,2,5,4,3,8,7,6,5,4) & 3 & 8 \\
$8_{11}$ & (1,2,3,5,4,7,8,6,7,8,1,2,3,5,4,6) & 2 & 8 $\bullet$ \\
$8_{12}$ & (1,2,4,6,5,3,1,2,4,6,5,3) & 2 & 6 \\
$8_{13}$ & (1,4,3,2,1,6,7,8,5,6,4,7,3,8,5,2) & 2 & 8 $\bullet$ \\
$8_{14}$ & (1,2,4,3,7,8,5,6,7,8,1,2,4,3,5,6) & 2 & 8 $\bullet$ \\
$8_{15}$ & (1,2,3,4,1,2,5,6,7,3,4,8,5,6,7,9,10,8,9,10) & 2 & 8 - 10 $\bullet$ \\
$8_{16}$ & (1,2,3,4,5,8,7,6,1,2,3,8,4,7,5,6) & 3 & 8 \\
$8_{17}$ & (1,5,6,7,8,4,3,2,1,5,4,6,7,3,8,2) & 3 & 8 \\
$8_{18}$ & (1,2,3,4,8,7,6,5,1,2,8,3,7,4,6,5) & 3 & 8 \\
$8_{19}$ & (1,3,2,1,10,9,8,7,6,5,4,10,9,8,7,3,6,5,4,2) & 3 & 8 - 10 \\
$8_{20}$ & (1,2,3,6,5,1,6,4,5,2,3,4) & 2 & 6 \\
$8_{21}$ & (1,2,3,5,6,1,2,4,5,6,3,4) & 2 & 6 \\
\hline
\end{tabular}

\begin{tabular}[h]{cccc}
Knot $K$ & Flat basket code of $K$ & $g(K)$ & $fpbk(K)$ \\
\hline
$9_1$ & (1,2,3,4,5,6,7,8,9,10,1,2,3,4,5,6,7,8,9,10) & 4 & 10 \\
$9_2$ & (1,2,4,1,2,10,7,8,9,10,7,8,5,6,9,5,3,4,6,3) & 1 & 6 - 10 $\bullet$ \\
$9_3$ & (1,3,2,1,10,9,8,7,6,5,4,3,10,9,8,7,6,5,4,2) & 3 & 10 $\bullet$ \\
$9_4$ & (1,3,4,1,3,10,5,6,7,8,9,10,5,6,7,8,2,4,9,2) & 2 & 8 - 10 $\bullet$ \\
$9_5$ & (1,2,10,7,6,2,10,3,8,7,3,4,9,8,4,9,5,1,6,5) & 1 & 6 - 10 $\bullet$ \\
$9_6$ & (1,2,3,4,5,6,9,10,7,8,9,10,1,2,3,4,5,6,7,8) & 3 & 10 $\bullet$ \\
$9_7$ & (1,2,3,4,6,7,8,5,6,9,10,7,8,9,10,1,2,3,4,5) & 2 & 8 - 10 $\bullet$ \\
$9_8$ & (1,5,3,4,1,2,6,7,8,5,2,6,7,8,3,4) & 2 & 8 $\bullet$ \\
$9_9$ & (1,2,3,4,5,9,10,6,7,8,9,10,1,2,3,4,5,6,7,8) & 3 & 10 $\bullet$ \\
$9_{10}$ & (1,3,2,1,10,9,8,10,9,7,6,5,4,3,8,7,6,5,4,2) & 2 & 8 - 10 $\bullet$ \\
$9_{11}$ & (1,8,7,6,3,5,4,2,1,8,7,6,3,5,4,2) & 3 & 8 $\bullet$ \\
$9_{12}$ & (1,2,3,7,4,6,1,2,3,7,4,8,5,6,8,5) & 2 & 8 \\
$9_{13}$ & (1,3,2,1,10,9,8,7,6,10,9,8,7,5,4,3,6,5,4,2) & 2 & 8 - 10 $\bullet$ \\
$9_{14}$ & (1,4,5,8,6,2,7,3,2,7,1,4,5,8,6,3) & 2 & 8 $\bullet$ \\
$9_{15}$ & (1,4,3,2,1,8,5,7,6,4,3,8,5,7,6,2) & 2 & 8 $\bullet$ \\
$9_{16}$ & (1,6,5,4,3,2,1,10,9,8,7,6,5,4,10,9,8,7,3,2) & 3 & 10 $\bullet$ \\
$9_{17}$ & (1,3,4,5,6,8,7,2,1,3,4,5,6,8,7,2) & 3 & 8 \\
$9_{18}$ & (1,2,3,4,1,2,3,5,6,7,8,4,5,6,7,9,10,8,9,10) & 2 & 8 - 10 $\bullet$ \\
$9_{19}$ & (1,7,2,3,1,7,2,4,5,8,6,3,4,5,8,6) & 2 & 8 $\bullet$ \\
$9_{20}$ & (1,2,3,4,8,5,6,7,1,2,3,4,8,5,6,7) & 3 & 8 \\
$9_{21}$ & (1,5,4,7,3,2,8,6,1,8,6,5,4,7,3,2) & 2 & 8 $\bullet$ \\
$9_{22}$ & (1,7,6,4,5,3,2,8,1,4,7,6,5,3,2,8) & 3 & 8 \\
$9_{23}$ & (1,2,3,4,1,2,3,5,6,7,8,4,5,6,9,10,7,8,9,10) & 2 & 8 - 10 $\bullet$ \\
$9_{24}$ & (1,2,4,7,3,6,5,8,1,2,7,6,4,3,5,8) & 3 & 8 * \\
$9_{25}$ & (1,3,4,8,1,9,2,3,5,9,2,6,10,7,5,6,10,7,4,8) & 2 & 8 - 10 \\
$9_{26}$ & (1,5,6,8,7,4,3,2,1,5,6,8,7,4,3,2) & 3 & 8 \\
$9_{27}$ & (1,5,6,7,4,3,2,8,1,5,6,7,4,3,2,8) & 3 & 8 \\
$9_{28}$ & (1,6,2,3,5,4,7,8,1,6,5,2,3,4,7,8) & 3 & 8 \\
$9_{29}$ & (1,2,7,3,5,4,6,8,7,1,2,3,5,6,8,4) & 3 & 8 ** \\
$9_{30}$ & (1,4,5,7,6,3,2,8,1,7,4,5,6,3,2,8) & 3 & 8 \\
$9_{31}$ & (1,4,5,6,3,2,7,8,1,4,5,6,3,2,7,8) & 3 & 8 \\
$9_{32}$ & (1,3,4,8,7,2,1,6,5,3,4,8,7,6,5,2) & 3 & 8 \\
$9_{33}$ & (1,3,7,4,6,5,2,8,1,7,6,3,4,5,2,8) & 3 & 8 \\
$9_{34}$ & (1,10,4,3,2,1,10,6,9,4,7,11,5,6,7,11,3,5,12,8,9,12,8,2) & 3 & 8 - 12 \\
$9_{35}$ & (1,3,7,1,8,2,4,8,2,6,9,5,6,9,10,3,4,5,10,7) & 1 & 6 - 10 $\bullet$ \\
$9_{36}$ & (1,8,7,6,3,2,5,4,1,8,7,6,3,5,4,2) & 3 & 8 \\
$9_{37}$ & (1,7,2,4,6,5,1,7,2,8,3,4,8,3,6,5) & 2 & 8 $\bullet$ \\
$9_{38}$ & (1,2,5,6,1,2,7,8,3,4,5,9,10,6,7,8,9,10,3,4) & 2 & 8 - 10 $\bullet$ \\
$9_{39}$ & (1,8,2,9,5,7,3,4,2,9,5,10,6,4,10,6,7,1,8,3) & 2 & 8 - 10 $\bullet$ \\
$9_{40}$ & (1,2,3,11,10,12,4,6,5,10,12,7,1,4,6,9,8,7,9,8,2,5,3,11) & 3 & 8 - 12\\
\hline
\end{tabular}

\begin{tabular}[h]{cccc}
Knot $K$ & Flat basket code of $K$ & g(K) & $fpbk(K)$ \\
\hline
$9_{41}$ & (1,5,8,2,7,8,2,4,3,6,7,9,5,4,3,6,10,1,9,10) & 2 & 8 -10 $\bullet$ \\
$9_{42}$ & (1,3,2,6,5,1,6,4,5,3,2,4) & 2 & 6 \\
$9_{43}$ & (1,5,4,3,2,1,5,4,9,10,8,7,6,9,10,8,3,7,2,6) & 3 & 8 - 10\\
$9_{44}$ & (1,2,3,5,1,6,8,7,6,8,4,5,2,7,3,4) & 2 & 6 - 8 \\
$9_{45}$ & (1,2,4,5,7,8,3,6,7,8,1,2,4,5,3,6) & 2 & 6 - 8 \\
$9_{46}$ & (1,3,6,1,4,5,3,2,4,6,2,5) & 1 & 6 $\bullet$ \\
$9_{47}$ & (1,3,2,7,6,5,4,8,1,3,7,6,2,5,4,8) & 3 & 8 \\
$9_{48}$ & (1,4,3,6,5,2,1,4,3,6,5,2) & 2 & 6 \\
$9_{49}$ & (1,4,3,10,9,8,7,6,10,9,5,8,7,4,3,6,2,1,5,2) & 2 & 8 -10 $\bullet$ \\
\hline
\end{tabular}
\end{center}

\vspace{15pt}

\subsection*{Acknowledgments}
This paper is a part of the master thesis of the second author. 
The authors wish to express their gratitude to Professors 
Takuji Nakamura and Yasutaka Nakanishi 
for encouragement and advices, 
Professor Mikami Hirasawa, for many discussion 
and informing his result on $9_{29}$, 
and Professor Tsuyoshi Kobayashi for informing of the result by 
his student, Mika Aoki. 

\end{document}